\documentclass[11pt]{amsart}
\usepackage{newlfont,amsmath,enumerate,amssymb}

\newcommand{\ind}{{\mathrm{ind}}}

\newcommand{\GL}{{\mathrm{GL}}}
\newcommand{\PGL}{{\mathrm{PGL}}}
\newcommand{\SO}{{\mathrm{SO}}}
\newcommand{\St}{{\mathrm{St}}}

\newtheorem{lemma}{Lemma}
\newtheorem{prop}[lemma]{Proposition}
\newtheorem{theorem}[lemma]{Theorem}
\newtheorem{cor}[lemma]{Corollary}

\title[Bessel models]{Iwahori component of Bessel model spaces}

\author[Chan]{Kei Yuen Chan}
\address{Department of Mathematics \\University of Georgia}
\email{KeiYuen.Chan@uga.edu}

\author[Savin]{Gordan Savin}
\address{Department of Mathematics \\University of Utah}
\email{savin@math.utah.edu}

\begin{document}
\begin{abstract} 
Let $k_0$ be a $p$-adic field of odd residual characteristic, and $G$ a special orthogonal group defined as acting on a split $2n+1$-dimensional orthogonal space $V$ over $k_0$. 
 Let $H$ be the Iwahori Hecke algebra of $G$. 
 A purpose of this short article is to compute the Iwahori component of a Bessel model space and identify it with an explicit 
projective $H$-module. 
 
\end{abstract} 

\maketitle

\section{Introduction}  

The Bessel models considered in this paper have a rather long history. They were introduced by Novodvorski and Piatetski-Shapiro in the case $n=2$ \cite{NPS}, 
extended by Bump, Friedberg and Furusawa \cite{BFF} and naturally appear in theta correspondences \cite{F}. 
Last but not least, these models appear as the second in a sequence of (generalized) 
Bessel models introduced by Gan, Gross and Prasad \cite{GGP}, where the first model is the Whittaker model space, also known as the Gelfand-Graev representation. 
A more precise definition is as follows. Let $V$ be a split $2n+1$-dimensional orthogonal space over $k_0$, a $p$-adic field of odd residual characterstic. 
In particular, we can write $V= X + X^{\vee} + k_0$, where 
$X$ is a totally isotropic subspace of dimension $n$.  Let $v_1, \ldots , v_n$  be a basis of $X$. 
Let $P=LN$ be a parabolic subgroup of $G=\SO(V)$  defined as the stabilizer of the partial flag $\langle v_1 \rangle \subset \ldots \subset \langle v_1, \ldots ,v_{n-1} \rangle$. Here $N$ 
is the unipotent radial and $L= (k_0^{\times})^{n-1} \times \PGL_2(k_0)$ is the Levi factor. The Levi factor $L$ acts naturally on characters of $N$, and the stabilizer of a 
generic character is a one dimensional torus $T_k \subset \PGL_2(k_0)$, where $k$ is a quadratic separable extension of $k_0$ such that $T_k\cong k^{\times}/k^{\times}_0$. 
We assume henceforth that $k$ is a field, so $T_k$ is compact. Let $\chi$ be the character of $T_kN$ obtained by extending the generic character of $N$ trivially to $T_k$. 
The (special) Bessel model space is $\mathrm{ind}_{T_kN}^G \chi$. In fact, if $k$ is ramified, then $T_k$ is a group scheme with two connected components over $O$, 
the ring of integers in $k_0$, so one can consider two characters $\chi^{\pm}$ of $T_KN$, trivial and non-trivial on the connected component of $T_k$. In this case 
we also have a (twisted) Bessel model space $\mathrm{ind}_{T_kN}^G \chi^-$. 

\smallskip 

Let $I$ be an Iwahori subgroup of $G$ and $H$ the Hecke algebra of compactly supported, $I$-biinvariant complex functions on $G$. As an abstract algebra, $H$ corresponds 
to the Coxeter diagram of type $\widetilde {C}_n$, but with unequal parameters:

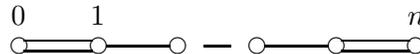
\begin{picture}(200,100)(-100,-15)

\put(53,30){\line(1,0){24}}
\put(22,32){\line(1,0){26}}
\put(22,28){\line(1,0){26}}
\put(90,30){\line(1,0){10}}
\put(113,30){\line(1,0){24}}
\put(142,32){\line(1,0){26}}
\put(142,28){\line(1,0){26}}

\put(20,30){\circle{6}}
\put(50,30){\circle{6}}
\put(80,30){\circle{6}}
\put(110,30){\circle{6}}
\put(140,30){\circle{6}}
\put(170,30){\circle{6}}

\put(17,38){$0$}
\put(47,38){$1$}
\put(167,38){$n$}

\end{picture}

\noindent 
More precisely,  $H$ is generated by elements $t_0, t_1, \ldots , t_n$, one corresponding to each vertex of the diagram, satisfying the braid relations as prescribed by the 
Coxeter diagram. These elements satisfy the quadratic relations $(t_i +1)(t_i -q)=0$ for all $i\neq 0$ and $(t_0+1)(t_0-1)=0$, where $q$ is the order of the residual field of $k_0$. 
The algebra $H$ has two finite dimensional subalgebras $H_0$ and $H_n$ obtained by removing the special vertices $0$ and $n$, respectively, of the Coxeter diagram. That is, 
$H_0$ is generated by $t_1, \ldots , t_n$ and $H_n$ is generated by $t_0, \ldots , t_{n-1}$. The algebra $H_0$ is the subalgebra of $H$ of functions supported on a hyper-special 
maximal compact subgroup of $G$ i.e. whose reduction mod $p$ is $\SO_{2n+1}$, while $H_n$ consists of functions supported on a maximal subgroup whose reduction mod $p$ is 
the disconnected group $\mathrm{O}_{2n}$.   Let $\mathrm{sgn}$ and $\mathrm{sgn}'$ be two one-dimenional characters (types) of $H_0$ defined by 
$t_i \mapsto -1$ for $i=1, \ldots, n-1$ and $t_n \mapsto -1$ and $t_n\mapsto q$, respectively. It is well known (\cite{BM} and \cite{R}) that for irreducible $G$-modules generated by $I$-fixed vectors 
the presence of the type $\mathrm{sgn}$ is equivalent to existence of a non-zero Whittaker functional. 
A similar relationship between the type $\mathrm{sgn}'$ and existence of a Bessel model was discovered Brubaker, Bump and Friedberg \cite{BBF}. 
 Following a suggestion of Sol Friedberg, as the first result of this paper, we prove that there exist isomorphisms of $H$-modules 
\[ 
(\mathrm{ind}_{T_kN}^G (\chi ))^I \cong H \otimes_{H_0} \mathrm{sgn}' \] \[  (\mathrm{ind}_{T_kN}^G (\chi^{\pm} ))^I \cong H \otimes_{H_n} \epsilon^{\pm}, 
\] 
where $k$ is an unramfied and ramified, respectively, quadratic extension of $k_0$.

\smallskip 
The second result of this paper concerns the restriction to $G=\SO_{2n+1}$ of the Steinberg representation $\St$ of the split orthogonal group $\SO_{2n+2}$. 
In fact, for the same price, 
one can consider two representations, the Steinberg representation $\St^{+}$, and its twist $\St^{-}$ by a quadratic character given as the composite of the spinor norm, and 
the unique non-trivial, unramified, quadratic character of $k_0^{\times}$. We prove that there exist isomorphisms of $H$-modules 
\[ 
(\St^{\pm})^I \cong H \otimes_{H_n} \epsilon^{\pm}. 
\] 
This result has several immediate consequences. If $\pi$ is an irreducible representation of $G$ generated by $I$-fixed vectors, then $\pi$ is a quotient of $\St$ if and only if 
it has a non-zero (special) Bessel model. Furthermore, the Iwahori component of $\St$ is projective, hence $\mathrm{Ext}_G^i(\St, \pi)=0$ for all $i>0$. 
 A similar result, projectivity of the Steinberg representation of $\GL_{n+1}$ when restricted to $\GL_n$, was obtained firstly in \cite{CS0} for the Iwahori component, 
 and then for all Bernstein components in \cite{CS3}. See the article of Prasad \cite{P} for a detailed discussion of ext-branching problems.

\section{Basic case}

Let $k_0$ be a $p$-adic field of odd residual characteristic,  $O$ the ring of integers in $k_0$, and $\varpi$ a uniformizer.
Let $G=\PGL_2 (k_0)$ and, (ab)using the $\GL_2$-terminology,  let $T$ be the torus of diagonal matrices and $B=TU$ the Borel subgroup of 
upper triangular matrices.  Let $\bar B = T \bar U$ be the Borel subgroup of lower triangular matrices. 
Then $G/B=\mathbb P^1(k_0)$ is a projective line. 
Let  $k$ be a quadratic extension of $k_0$ and $T_k=k^{\times}/k_0^{\times}$. The torus $T_k$ acts simply transitively on the projective line 
 $\mathbb P^1(k_0)\cong k^{\times}/k _0^{\times}$, hence we have exact decompositions
\[ 
G= T_kB  = T_k \bar B. 
\] 
Let $K$ be a maximal compact subgroup of integral matrices, and $I$ the Iwahori subgroup of $K$ such that $I\cap U$ has the off-diagonal entry divisible by $\varpi$. 
Then $K= I \cup Is_1 I$ where $s_1$ is a permutation matrix.  Let  $\tilde I = I \cup s_0 I$ be the normalizer of $I$ in $G$. The Iwahori Hecke algebra is generated by 
two elements $t_0$ (supported and equal to 1 on $s_0 I$ ) and $t_1$ (supported and equal to 1 on $Is_1 I$) 
satisfying relations $t_0^2=1$ and $(t_1+1)(t_1-q)=0$. Let $H_0$ and $H_1$ be the 2-dimensional subalgebras generated by 
$t_0$ and $t_1$, respectively. We have Bernstein decompositions  
\[ 
H \cong A\otimes H_1\cong A\otimes H_0 
\] 
where $A\cong \mathbb C[T/T(O)]$. This isomorphism satisfies the following properties. 
The group $T/T(O)$ is isomorphic to the co-character lattice of $T$, and for every $x \in T/T(O)$ let $\theta_x\in A$ be  
the element corresponding to $x$ under the isomorphism $A\cong \mathbb C[T/T(O)]$.  If $(\pi,V)$ is a  smooth $G$-module then we have a natural 
isomorphism $\pi^I \cong \pi_{\bar U}^{T(O)}$. This isomorphism intertwines the action $t_x$ on $\pi^I$ with the action of $x$ on $\pi_{\bar U}^{T(O)}$.

The algebra $H_0$ has 2 characters, denoted by $\epsilon^{+}(t_0)=1$ and $\epsilon^{-}(t_0)=-1$, while $H_1$ has two  
characters $\mathrm{sgn}(t_1)=-1$ and $\mathrm{sgn}'(t_1)=q$. Inducing these characters to $H$ we get four projective $H$-modules. 

\smallskip 

Assume now that $k$ is unramified.  The Bessel model space is 
\[ 
\Pi =\ind_{T_k}^{G} (1)=C_c^{\infty} (T_k\backslash G).
\] 
Since $k$ is unramified, we can assume that $T_k$ sits in $K$. 
Thus the characteristic function of $K$ is contained in $\Pi$, and $t_1$ acts on it by the character $\mathrm{sgn}'$. Hence, by the Frobenius reciprocity, 
 we have a morphism of $H$-modules 
\[ 
H\otimes_{H_1} \mathrm{sgn}' \rightarrow \Pi^I  
\] 
where $1\otimes 1$ is mapped to the characteristic function of $K$. 
We shall prove that this is an isomorphism by proving that it is so as $A$-modules. 
 From the  decomposition $G =T_k \bar B$,  it easilly follows that the space of $T(O)$-fixed vectors in the Jacquet module $\Pi_{\bar U}$ is isomorphic to $A\cong \mathbb C[T/T(O)]$. 
Furthermore, from the decomposition  $K=T_k (\bar B \cap K) $, it follows that the characteristic function of $K$ maps to 
$1\in \mathbb C[T/T(O)]$ under the Jacquet module isomorphism.  Now the claimed isomorphism follows from the Bernstein decomposition of $H$. 

As a consequence, we get that the Steinberg representation $(t_0\mapsto -1, t_1\mapsto -1)$ does not have the (unramified) Bessel model, observed by Waldspurger many years ago. 

\smallskip  

Assume now that $k$ is ramified. The image of valuation on $k$ is $\frac12\mathbb Z$, hence  $T_k$ has two unramified characters $\chi^{+}$, the trivial, and 
$\chi^-$ which takes value $-1$ on any uniformizing element in $k$. A (twisted) Bessel model space is 
\[ 
\Pi^{\pm} =\ind_{T_k}^{G} (\chi^{\pm}), 
\] 
i.e. $\Pi^+$ is the Bessel model, while $\Pi^-$ is the twisted. In this case $T_k \subset \tilde I$ and $\chi^{\pm}$ can be extended to characters of $\tilde I$, trivial on $I$. 
These functions, supported on $\tilde I$ are elements in $\Pi^{\pm}$ and we have isomorphisms of $H$-modules 
\[ 
(\Pi^{\pm})^I\cong H\otimes_{H_0} \mathrm{sgn}^{\pm}. 
\] 

\smallskip 
\noindent 
{\bf Remark:} 
Now assume that $G=D^{\times} /k_0^{\times}$ where $D$ is a quaternion algebra over $k_0$. 
 If we normalize the valuation on $D^{\times}$ so that the image is $\mathbb Z$, then 
the restriction of valuation to $k_0$ has the image $2\mathbb Z$. Thus the valuation gives a homomorphism of $G$ on $\mathbb Z/2\mathbb Z$. The kernel of this map is the 
Iwahori subgroup, hence the Iwahori Hecke algebra is the group algebra of $\mathbb Z/2\mathbb Z$. Its representations correspond to the Steinberg and twisted Steinberg representation  by the 
Jacquet-Langlands correspondence. The image of valuation on $k$ is $2\mathbb Z$ or $\mathbb Z$ if $k$ is unramfied and ramified, respectively. Thus only the trivial representation 
has the Besel model if $k$ is ramified. By Gross-Prasad conjectures (a theorem), the unramified twist of of the Steinberg representation of $\PGL_2(k_0)$ 
$(t_0\mapsto 1, t_1\mapsto -1)$ has a (ramified) Bessel model, which fits with all of the above.

\section{Hecke algebra of odd split orthogonal groups}

This is mostly taken from \cite{GS}. Let $V$ be a split $2n+1$-dimensional orthogonal space over $k_0$. 
In particular, we can write $V= X + X^{\vee} + k_0$, where 
$X$ is a totally isotropic subspace of dimension $n$.  A basis $v_1, \ldots , v_n$ of $X$ gives a maximal split torus $T= (k_0)^{\times}$ in $\SO(V)$. 
Let $\Sigma$ be the corresponding root system of type $B_n$. 
We realize the root sytem in $E=\mathbb R e_1 + \ldots + \mathbb R e_n$, so that 
$\Sigma=\{\pm e_i\pm e_j\}\cup  \{\pm e_i\}$. 
Fix a set $\Delta$ of simple roots consisting of
roots $\alpha_1=e_1-e_2, \ldots ,\alpha_{n-1}=e_{n-1}-e_n$ and $\alpha_n=e_n$. The choice of simple roots gives us a standard Borel subgroup $B=TU$ and 
its opposite $\bar B=T \bar U$. One can think of $B$ as stabilizing the partial flag $\langle v_1 \rangle \subset \ldots \subset \langle v_1, \ldots ,v_n \rangle$. 

\smallskip 

Roots can be viewed as functionals on $E$, using the standard dot product on $E$. 
Affine roots are functionals $\alpha+m$ where $\alpha\in\Sigma$ and $m\in\mathbb Z$.
We have a set of simple affine roots $\Delta_a=\Delta\cup\{\alpha_0\}$ where
$\alpha_0=1-e_1-e_2$.  
The affine Weyl group $W_a$ (of type $\mathrm B_n$) is generated by reflections about the root hyperplanes. The connected components of 
the complement in $E$ of the union of the root hyperplanes are called chambers.
 Let $C$ be the chamber consisting of all $x\in E$ such that $\alpha_i(x)>0$ for all $i=0, \ldots, n$. The group $W_a$ 
 is generated by reflections  $s'_0$ and $s_1, \ldots ,s_n$ corresponding to
the simple affine roots satisfying the braid relations given by the following Coxeter diagram:

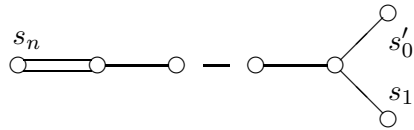
\begin{picture}(200,100)(-100,-15)

\put(53,30){\line(1,0){24}}
\put(22,32){\line(1,0){26}}
\put(22,28){\line(1,0){26}}
\put(90,30){\line(1,0){10}}
\put(113,30){\line(1,0){24}}
\put(142,32){\line(1,1){16}}
\put(142,28){\line(1,-1){16}}

\put(20,30){\circle{6}}
\put(50,30){\circle{6}}
\put(80,30){\circle{6}}
\put(110,30){\circle{6}}
\put(140,30){\circle{6}}

\put(160,50){\circle{6}}
\put(160,10){\circle{6}}

\put(18,38){$s_n$}
\put(160,36){$s_0'$}
\put(160,17){$s_1$}

\end{picture}

 The affine Weyl group is a semi direct product of the Weyl group $W$ 
(generated by $s_1, \ldots ,s_n$) and a normal
subgroup consisting of translations by $\lambda=(m_1, \ldots , m_n)\in \mathbb Z^n$
such that $\sum_{i=1}^n m_i$ is even. The extended affine Weyl group $W_e$, relevant
to the adjoint group $\SO(V)$, is defined by enlarging the group of translations to full $ \mathbb Z^n \cong T/T(O)$.  It turns out that $W_e$ is also a Coxeter group. 
Indeed, $W_e=W_a \cup s_0 W_a$ where 
$s_0$ is the  reflection about the hyperplane $x_1=\frac{1}{2}$. One checks that 
\[
s'_{0}=s_0 s_{1} s_0 .
\]
Now it is easy to see that $W_e$ is isomorphic to the affine Weyl group
of type $\mathrm C_n$, since it is generated by
reflections $s_0, \ldots, s_n$ and braid relations corresponding to the following
Coxeter diagram:

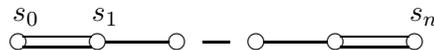
\begin{picture}(200,100)(-100,-15)

\put(53,30){\line(1,0){24}}
\put(22,32){\line(1,0){26}}
\put(22,28){\line(1,0){26}}
\put(90,30){\line(1,0){10}}
\put(113,30){\line(1,0){24}}
\put(142,32){\line(1,0){26}}
\put(142,28){\line(1,0){26}}

\put(20,30){\circle{6}}
\put(50,30){\circle{6}}
\put(80,30){\circle{6}}
\put(110,30){\circle{6}}
\put(140,30){\circle{6}}
\put(170,30){\circle{6}}

\put(18,38){$s_0$}
\put(48,38){$s_1$}
\put(168,38){$s_n$}

\end{picture}

Let $I$ be the Iwahori subgroup of $\SO(V)$ corresponding to the chamber $C$ i.e. it is generated by $T(O)$ and the affine root spaces $U_{\alpha}$ for all affine roots $\alpha$ such 
that $\alpha >0$ on $C$.  The $I$-double cosets in $\SO(V)$ are parameterized by $W_e$. 
Let $l_0$ be the weighted length function on $W_e$ such that $l_0(s_0)=0$ and
it is 1 on other simple reflections.
 Then $[IwI:I]= q^{l_0(w)}$ for every $w\in W_e$.

Let $H$ denote the  Iwahori Hecke algebra of $\SO(V)$. For any $w\in W_a$
let $t_w\in H^+$ be the characteristic function of the double coset parameterized by
$w$. For simplicity, let $t_i$ denote $t_{s_i}$. Then $t_i$
satisfy quadratic relations $t_0^{2}=1$ and $(t_{i}-q)(t_{i}+1)=0$, if $i\neq 0$, and
braid relations given by the above Coxeter diagram. In fact, as an abstract
algebra, $H$ is generated by $t_{0}, \ldots, t_{n}$ modulo these
quadratic and braid relations.  Let $H_0$ and $H_n$ be the finite dimensional subalgebras generated by all $t_i$ except $t_n$ and $t_0$, respectively. 
We have  Bernstein decompositions 
\[ 
H \cong A\otimes H_0 \cong A \otimes H_n. 
\] 
where $A$ is isomorphic to $\mathbb C[T/T(O)]$ and the isomorphism has the same property with respect to the Jacquet functor as in the rank one case $(G\cong \PGL_2(k_0))$.

Then $H_0$ and $H_n$ each have 2 characters $\epsilon^{+}, \epsilon^-$ and $\mathrm{sgn}$, $\mathrm{sgn}'$, respectively, 
 such that $t_i\mapsto -1$ for all $i\neq 0,n$, with the rest of definition as in the rank one case. We also have four projective $H$-modules 
obtained by inducing these characters to $H$.

 \section{Special Bessel Models}

 Let $P=LN$ be a standard parabolic subgroup corresponding to the short root $e_n$, with the Levi $L= (k_0^{\times})^{n-1} \times \PGL_2$,
  where $\PGL_2$ corresponds to the short root $e_n$. 
 In other words, the group $P$ is defined as the stabilizer of the partial flag $\langle v_1 \rangle \subset \ldots \subset \langle v_1, \ldots ,v_{n-1} \rangle$. 
 In particular, $N/[N,N]$ is a vector space on which $L$ acts. 
 The action of the $\PGL_2$ factor is trivial on the subspace spanned by the lines corresponding to the roots $e_1-e_2, \ldots , e_{n-2}-e_{n-1}$,  and an 
 irreducible representation on the 3-dimensional subspace spanned by the lines corresponding to the roots $e_{n-1} - e_n, e_{n-1}, e_{n-1} + e_n$. Let 
 $x_i$ be the pinning coordinates for the roots $e_{i} - e_{i+1}$, and $x_n$  for the root $e_{n-1}+e_n$. 
 (The pinnings are chosen so that the subgroup generated by $T(O)$ and the affine root groups $U_{\alpha}$ for all $\alpha$ such that $\alpha(0) \geq 0$ is a hyper-special 
 maximal compact subgroup.) 
 For every $a\in F^{\times}$ define a character 
 \[ 
 \psi_a : N \rightarrow \mathbb C^{\times} 
 \] 
 by 
 \[ 
 \psi_a(n) =\psi( x_1 + \ldots + x_{n-1} + a x_n) 
 \] 
 where $\psi: k_0 \rightarrow \mathbb C^{\times}$ has conductor $\varpi O$. 
 The stabilizer of $\psi_a$ in $L$ is a torus $T_k\subset \PGL_2$.  We choose $a=u \in O^{\times}$ so that $k$ is the unramified extension of $k_0$.  If $a=r$ is a 
 uniformizing element in $O$, then $k$ is ramified. Abusing notation, we extend $\psi_u$ to  a character $\chi$ of $NT_k$, trivial on $T_k$. 
 We also have two characters $\chi^{\pm}$ of $NT_k$ in the 
 ramified case, due to disconnectedness of $T_k$.  For sake of uniform exposition, let $\chi$ denote any of these characters of $T_k N$, for a moment.  The 
 Bessel space is 
 \[ 
 \Pi= \ind_{NT_k}^G( \chi).
 \]  
 
 \begin{lemma}  As $T$-modules, $\Pi_{\bar U}^{T(O)} \cong \mathbb C[T/T(O)]$. 
 \end{lemma} 
 \begin{proof}  Let $W_n\subset W$ consist of $w\in W$ such that $w^{-1}(e_n)>0$. We have a Bruhat decomposition 
 \[ 
 G= P W_n \bar B = N L W_n \bar B = NT_k  W_n \bar B 
 \] 
 where for the second equality we used that $\PGL_2= T_k \cdot (\PGL_2 \cap \bar B)$.  It follows that $\Pi$, as a $\bar B$-module, has a filtration arising from the 
 geometry of the Bruhat decomposition with subquotients $\Pi_w$ parameterized by $w\in W_n$. If $w\neq 1$, then there exists a simple root $\alpha$, 
 necessarily different from $e_n$, such that $w^{-1}(\alpha) <0$. 
 Now it is easy to check \cite{CS1} that $(\Pi_w)_{\bar U}=0$. The bottom piece of the filtration is isomorphic to 
 $C^{\infty}_c(T_k\backslash \bar P)$.  It is clear that 
 \[ 
 (C^{\infty}_c(T_k\backslash \bar P))_{\bar N} \cong C^{\infty}_c(T_k\backslash L)\cong  C^{\infty}_c((k_0^{\times})^{n-1}) \otimes  C^{\infty}_c(T_k\backslash \PGL_2).
 \]  
 Now lemma follows from the $\PGL_2$-computation. 
 \end{proof} 
 
 In view of the isomorphism of $A\cong C[T/T(O)]$-modules $\Pi^I \cong \Pi_{\bar U}^{T(O)}$, it follows that $\Pi^I$ is a free $A$-module of rank one. In order to determine the 
 $H$-structure, we need to compute the action of the finite subalgebra $H_0$ or $H_n$ on a generator of the $A$-module. 
 We shall do this by carefully picking the Iwahori subgroup. Recall that the chamber $C$ was defined by 
  $\alpha_i >0$ for all simple roots $\alpha_i$, $i=0, \ldots , n$. 
  
  \subsection{Unramified case} We let $C_u=-C$ and, abusing notation, let $I$ be the Iwahori subgroup generated by $T(O)$ and affine root groups $U_{\alpha}$ for 
  all $\alpha >0$ on $C_u$. Note that previously defined simple reflections $s_1, \ldots , s_{n-1}, s_n$, corresponding to the roots,  
  $e_1-e_2, \ldots ,e_{n-1}-e_n, e_n$ are still simple for this choice of the chamber. In particular, we have 
  a finite subalgebra $H_n$ generated by $t_i$, the characteristic functions of $Is_iI$, for $i=1, \ldots , n$. Let $\mathrm{sgn}'$ be a character of $H_n$ such that 
  $t_n \mapsto q$ and $t_i\mapsto -1$ for $i\neq n$. 
  
  \begin{theorem} Let $\Pi$ be the unramified Bessel model space.  Then $\Pi^I \cong H\otimes_{H_n} \mathrm{sgn}'$ as $H$-modules. 
  \end{theorem} 
  \begin{proof} 
   Let $K= I \cup Is_n I$. Then $T_k \subset K$, in fact, note that $K=T_k I$. Hence 
  we have a decomposition 
  \[ 
  NT_k I= NK = N (K\cap L) (K\cap \bar N) 
  \] 
  where the last factorization is a direct product as sets. Since the choice of $I$ guarantees that $\psi_u$, when restricted to $I\cap N$ is trivial, we have a unique function 
  $f_0\in  \Pi^I$  supported on $NT_k I$ such that 
  $f_0(n l \bar n ) = \psi_u(n)$  for all $n\in N$, $l\in K\cap L$ and $\bar n \in K\cap \bar N$.  Note 
   that $f_0$ sits in the bottom piece of the $\bar B$-filtration of $\Pi$ and under the isomorphism 
  $\Pi^I\cong \mathbb C[T/T(O)]$ it maps to the characteristic function of $T(O)$. We now need the following lemma: 
 
  \begin{lemma} For every $i=1, \ldots, n$, let $t_i \ast f_0$ denote the action of $t_i \in H_n$ on $f_0\in \Pi^I$. Then 
  $t_n\ast f_0= q f_0$ and $t_i \ast f_0 =- f_0$ for $i\neq n$. 
  \end{lemma} 
 \begin{proof} Since $f_0$ is right $K$-invariant, for $K=I\cup Is_n I$, it follows at once that $t_n \ast f_0= qf_0$. We now use the decomposition 
 $G=U W_e I$ and observe that $NK \subset UI \cup Us_n I$. It follows that $f_0(s_i)=0$ for all $i\neq n$. This will be of use to compute the action of 
 $t_i$ on $f_0$, for $s_i$ the reflection corresponding to $\alpha_i=e_i-e_{i+1}$. We shall simplify notation, in order to avoid writing the subscript $i$ many times, and 
 will write $s=s_i$ and  $\alpha=\alpha_i$. Note that the affine root space $U_{\alpha}$ is not contained in $I$, but $U_{-\alpha}$ and $U_{\alpha+1}$ are.  
 Let $x_{\beta} : O \rightarrow U_{\beta}$ denote a pinning for any affine root $\beta$. Then 
 \[ 
 IsI= \cup_{t\in O/\varpi O} x_{-\alpha}(t) sI. 
 \] 
 Recall that 
 \[ 
 (t_i \ast f_0)(g) =\int_{IsI} f_0(gh) ~dh. 
 \] 
 Thus, if $(t_i \ast f_0)(g)\neq 0$, then there exists $h\in IsI$ such that $gh\in NT_kI$. Since $h^{-1}\in IsI$, it follows that  
 \[ 
 g\in NT_kIsI =  \cup_{t\in O/\varpi O} N T_k x_{-\alpha}(t) sI. 
 \] 
There are two cases to discuss. If $t\in \varpi O$ then $N T_k x_{-\alpha}(t) sI= NT_k sI$. If $t \in (O/ \varpi O)^{\times}$, then we have an identity 
 $x_{\alpha}(1/t) x_{-\alpha}(-t) \equiv x_{-\alpha}(t) s \pmod{T(O)}$, and 
 \[ 
 N T_k x_{-\alpha}(t) sI = N T_k x_{\alpha}(1/t)x_{-\alpha}(t)I = NT_kI. 
 \] 
  We have shown that  $t_i\ast f_0$ is supported on 
 $NT_kI$ and $NT_k sI$ so it suffices to compute the value of $t_i\ast f_0(g)$  for $g=1$ and $s$. To that end, 
 \[ 
 (t_i\ast f_0)(s) = \sum_{t\in O/\varpi O} f_0(s x_{-\alpha}(t) s) =\sum_{t\in O/\varpi O} f_0( x_{\alpha}(t) ) =\sum_{t\in O/\varpi O} \psi(t) =0, 
 \] 
 \[ 
 (t_i\ast f_0)(1) = \sum_{t\in O/\varpi O} f_0( x_{-\alpha}(t) s) =\sum_{t\in (O/\varpi O)^{\times}} f_0(x_{\alpha}(1/t) x_{-\alpha}(-t)) =\sum_{t\in (O/\varpi O)^{\times}} \psi(1/t) =-1. 
 \] 
 We note that the second identity follows from the fact that $f_0(s)=0$. The lemma is proved. 
 \end{proof} 
 Now it is easy to prove the theorem. 
 The Frobenius reciprocity implies that we have a natural map from $H\otimes_{H_n} \mathrm{sgn}'$ to $\Pi^I$ that sends $1\otimes 1$ to $f_0$. This map is an isomorphism since $H\otimes_{H_n} \mathrm{sgn}' \cong \Pi^I$ as $A$-modules. 
 
 \end{proof}  
 
  \subsection{Ramified case} Let $w$ be the permutation defined by $i\mapsto n+1-i$ for all $i=1, \ldots, n$. 
  We let $C_r=w(C)$ and, abusing notation, let $I$ be the Iwahori subgroup generated by $T(O)$ and affine root groups $U_{\alpha}$ for 
  all $\alpha >0$ on $C_r$.  Note that previously defined simple reflections $s_1, \ldots , s_{n-1}$, corresponding to the roots,  
  $e_1-e_2, \ldots , e_{n-1}-e_n$ are still simple for this choice of the chamber. Let $s_0$ be the reflection corresponding to the affine functional $e_n-1/2$. 
  Observe that $s_0$ is not a root reflection, however, it normalizes the chamber $C_r$.  Let $H_0$ be 
  the finite algebra generated by $t_i$, the characteristic functions of $Is_iI$, for $i=0, \ldots , n-1$.  Note that $t_0^2=1$. 
  Let $\mathrm{\epsilon}^{pm}$ be two characters of $H_0$ such that 
  $t_0 \mapsto \pm 1$ and $t_i\mapsto -1$ for $i\neq 0$. 
  
  \begin{theorem} Let $\Pi^{\pm}$ be the two ramified Bessel model spaces.  Then $(\Pi^{\pm})^I \cong H\otimes_{H_0} {\epsilon}^{\pm} $ as $H$-modules. 
  \end{theorem} 
  \begin{proof} We just give a sketch of the proof. 
   In this case $s_0$ normalizes $I$ and let $K= I \cup s_0 I$. Then $T_k \subset K$,  $K=T_k I$ and  
    we have a decomposition 
  \[ 
  NT_k I= NK = N (K\cap L) (K\cap \bar N) 
  \] 
  where the last factorization is a direct product as sets. The choice of $I$ guarantees that $\psi_r$, when restricted to $I\cap N$ is trivial. 
  Let $f^{+}_0\in (\Pi^{+}) ^I$  supported on $NT_kI$ such that 
  $f^+_0(n l \bar n ) = \psi_u(n)$ for all $n\in N$, $l\in K\cap L$ and $\bar n \in K\cap \bar N$, and $f^{-}_0\in (\Pi^{-}) ^I$  such that 
  $f^-_0(n l \bar n ) = \pm \psi_u(n)$ where the sign depends whether $l \in K\cap L$ is in $I\cap L$ or not.  The proof now proceeds in the same way 
  as in the unramified case, by showing that $f_0^{\pm}$ are eigenfunctions for $H_0$. We leave details to the reader. 
  
  \end{proof} 
 
 \section{Steinberg representation} 
 
 Let $G$ be a split reductive group over $k_0$. 
 Let $\mathrm{St}$ be the Steinberg representation of $G$.  Let $B$ be a Borel subgroup of $G$, 
$\bar U$ the unipotent radical of $\bar B$, the Borel opposite to $B$, and $X_w=B w\bar U$ are the Bruhat cells. Write $X=B\bar U$ for the open cell. 
 For any subset $J$ of simple roots $\Pi$, let $P_J$ be the standard parabolic subgroup associated to $J$ (and containing $B$).
In particular, $P_{\emptyset}=B$.  Let $C^{\infty}_c(P_J \setminus G)$ be the space of compactly supported smooth $P_J$-invariant functions on $G$. 
Let $\mathrm{St}$ be the Steinberg representation of $G$.  
We use the following realization of the Steinberg representation: 
\[   \mathrm{St} = C^{\infty}_c(B\setminus G) / \sum_{ \emptyset  \neq J \subset \Pi} C^{\infty}_c(P_J \setminus G) .
\]
Thus we have a $\bar B$-equivariant map $\Omega: C_c^{\infty}(B\setminus X)\rightarrow\mathrm{St}$ given as the composition of natural maps 
\begin{align} \label{eqn cont steinb}
    C_c^{\infty}(B\setminus X) \rightarrow  C^{\infty}_c(B\setminus G) \rightarrow \mathrm{St}. 
\end{align}


\begin{prop} \label{prop steinberg} 
The map $\Omega$ is a $\bar B$-equivariant isomorphism of $C_c^{\infty}(B\setminus X)$ and $\mathrm{St}$. 
\end{prop}

\begin{proof}  Let $\mathbb C[W]$ denote the space of functions on $W$. Consider it as a $W$-module for the action by right translations. For every 
simple root $\alpha$, let $W_{\alpha}=\{1,s_{\alpha}\}$. Then $\mathbb C[W_{\alpha}\backslash W]$  is a submodule of $\mathbb C[W]$ consisting of left 
$W_{\alpha}$-invariant functions. For injectivity we need the following lemma. 

\begin{lemma} Let $\delta\in \mathbb C[W]$ be the delta function corresponding to the identity element. Then $\delta$ cannot be written as a linear combination 
of elements in $\mathbb C[W_{\alpha}\backslash W]$ where $\alpha$ runs over all simple roots. 
\end{lemma}
\begin{proof} Functions in $\mathbb C[W_{\alpha}\backslash W]$ are perpendicular to the sign character. Hence any linear combination of such functions is also 
perpendicular to the sign character. But $\delta$ is not, hence lemma. 
\end{proof} 
We can now prove injectivity of $\Omega$. Let $f\in C_c^{\infty}(B\setminus X)$ be in the kernel of $\Omega$. 
Then there exist $f_{\alpha}\in C^{\infty}_c(P_{\alpha} \setminus G)$ such that $f=\sum_{\alpha\in\Pi} f_{\alpha}$.  For every $\bar u \in \bar U$, the function 
$w\mapsto f_{\alpha}(w\bar u)$ is in $\mathbb C[W_{\alpha}\backslash W]$. On the other hand, $w\mapsto f(w\bar u)$ is a multiple of $\delta$.   
Lemma implies that $f(\bar u)=0$. 

For surjectivity, let $V_r\subseteq  C_c^{\infty}(B \setminus G)$ be the subspace of functions supported on the union of the Bruhat cells $X_w$ for $w\in W$ such that $l(w) \leq r$. 
Let $V_w=C_c^{\infty}(B \setminus X_w)$. 
Then, if $r>1$, we have an exact sequence 
\[ 
0 \rightarrow V_{r-1} \rightarrow V_r \rightarrow \bigoplus_{l(w)=r} V_w \rightarrow 0. 
\] 
Let $v\in \mathrm{St}$  be the image of $f\in V_r$. We need to show that $v$ is the image of some $f'\in V_{r-1}$. For every $w$ such that $l(w)=r$, pick $f_w\in V_r$ 
supported on $X_{w'}$ for $l(w) < r$ and $X_w$. Then $f-\sum_{l(w)=r} f_w \in V_{r-1}$. Since $r>1$, for every $w$ such that $l(w)=r$, there exists a simple root $\alpha$ 
such that $l(s_{\alpha} w) =r-1$.  
The group $G$ has a cell decomposition as a union of $Y_w=P_{\alpha} w \bar U$ where $w$ runs over all $w\in W$ such that $l(s_{\alpha} w) =l(w) -1$.  Note that 
$B\backslash X_w=P_{\alpha} \backslash Y_w$ for such $w$. Going back to our fixed $w$ such that $l(w)=r$, there exists a function
 $h_w \in C_c^{\infty}(P_{\alpha} \setminus G)$ such that the support of $h_w$ is on $Y_w$ and larger orbits, and $h_w=f_w$ on $B\backslash X_w=P_{\alpha} \backslash Y_w$. 
 The support of $h_w$, viewed as an element of $C_c^{\infty}(B \setminus G)$, is contained in $X_w$ and the union of $X_{w'}$ such that $l(w') < l(w)$. 
 Hence $f'=f-\sum_{l(w)=r} h_w \in V_{r-1}$ and $f'$ has the image $v$ in $\mathrm{St}$. Hence $\Omega$ is surjective.  
\end{proof}

Let $\mathrm{ch}_{I}$ be the characteristic function of $B(\bar U \cap I)$. 
Since $I=(B\cap I)(\bar U \cap I)$, it is an $I$-fixed element in $C_c^{\infty}(B \setminus G)$. Hence 
$v_0=\Omega(\mathrm{ch}_{I})$ spans the line of $I$-fixed vectors in $\mathrm{St}$. In other words, after switching the roles of $B$ and $\bar B$ we have the following:  

\begin{cor} \label{C:iwahori_vector} 
As a $B$-module, the Steinberg representation is isomorphic to $\ind_T^B(1) \cong C_c^{\infty}(U)$. This isomorphism maps  a non-zero $I$-fixed vector to
the characteristic function of $I\cap U$. 
\end{cor}

 \section{Restricting Steinberg}

  Let $G'$ be a split orthogonal group of the type $D_{n+1}$, and $T'$ a maximal split torus of  $G'$.  
 Let $\Sigma'$ be the corresponding root system. 
 The standard realization of corresponding root system is in   $E'=\mathbb R e_1 + \ldots + \mathbb R e_{n+1}$, so that 
$\Sigma'=\{\pm e_i\pm e_j\}$. 
Let  $\Delta'$ be a set of simple roots consisting of
roots $\alpha_1=e_1-e_2, \ldots ,\alpha_{n}=e_{n}-e_{n+1} $ and $\alpha_{n+1}=e_n+e_{n+1}$. Let $s_1, \ldots s_{n+1}$ be the corresponding simple reflections. 
 The choice of simple roots gives us a standard Borel subgroup $B'=T'U'$ and 
its opposite $\bar B'=T' \bar U'$.   Let $W'$ denote the Weyl group. 
We have an involution $\sigma$ of $\Delta'$ that permutes $\alpha_n$ and $\alpha_{n+1}$. This involution lifts to an involution 
of $G'$ such that the group of fixed points is $G$, the split odd orthogonal group. Let $T\subset T'$ and $U\subset U'$ be the subgroups of $\sigma$-fixed elements. 
Since $T$ contains a strongly $G'$-regular element, the centralizer of $T$ in $G'$ is $T'$. This implies that the normalizer $N_G(T)$ is contained in $N_{G'}(T')$  and 
we have an identity $N_G(T)= (N_{G'}(T'))^{\sigma}$. The quotients $N_G(T)/T(O)$ and $N_{G'}(T')/T'(O)$ are isomorphic to extended affine Weyl groups 
$W_e$ and $W'_e$ and our next task is to describe the identity $N_G(T)= (N_{G'}(T'))^{\sigma}$ on the level of extended Weyl groups. 

\smallskip 

Consider the affine root $\alpha_0= 1-e_1-e_2$. Then $\Delta'_a=\Delta'\cup \{\alpha_0\}$ is a set of simple affine roots, and let 
 $C' \subset E'$ be the corresponding chamber. The extended  affine Weyl group $W'_e$ 
is a semi-direct product of $W'$ and the group $\mathbb Z^{n+1}$ of translations of $E'$. Then 
$W'_e= W'_a \cup s_0 W'_a$ where  $W_a'$ is the affine Weyl group and $s_0$ is a unique non-trivial element in $W'_e$ that stabilizes $C'$. 
More precisely, $s_0$ is a central symmetry about the point $(1/2, 0)$ in the two dimensional subspace spanned by $e_1$ and $e_{n+1}$.  
The extended affine Weyl group is generated by $s_0, s_1,  \ldots , s_{n+1}$. The involution $\sigma$ descends to $W_e'$ and 
$(W_e')^{\sigma}$ consists of elements commuting with the linear map on $E'$ defined by $e_{n+1} \mapsto -e_{n+1}$ and identity on $E\subset E'$.  
Note that $s_0, \ldots, s_{n-1} \in (W_e')^{\sigma}$, and we have a natural isomorphism $(W_e')^{\sigma} \cong W_e$ is obtained by restricting the action of 
$w\in (W_e')^{\sigma}$ to $E$. Note that under this isomorphism $s_0, \ldots, s_{n-1}\in W_e'$ map to the elements of $W_e$ denoted by the same symbols. 

\smallskip

 Let $I'\subset G'$ be the Iwahori subgroup corresponding to $C'$ and $H'$ the corresponding 
Iwahori Hecke algebra. This algebra is generated by $t'_i$, the characteristic functions of $I's_i I$, for $i=0, \ldots, n+1$. The Steinberg representation of $G'$ is the 
unique irreducible representation $\mathrm{St}$ such that $\mathrm{St}^{I'}=\mathbb C v_0$,   $t_i \ast v_0 = -v_0$ for $i=1, \ldots , n+1$, and $t_0 \ast v_0= v_0$. 
Let $C=E\cap C'$, and $I$ the corresponding Iwahori subgroup, and $H$ the Hecke algebra, generated by $t_i$, $i=0, \ldots, n$.  Let $H_0\subset H$ be the 
finite algebra generated by $t_i$, $i=0, \ldots, n$, and $\epsilon^+$ the character of $H_0$ such that $t_0\mapsto 1$ and $t_i\mapsto -1$, for $i=1, \ldots, n-1$. 

\begin{lemma} \label{L:character} 
 Let $\mathrm{St}$ be the Steinberg representation of $G'$ and $v_0\in \mathrm{St}^{I'}$. Then $H_0$ acts on $v_0$ by the character $\epsilon^+$. 
\end{lemma} 
\begin{proof} For $i=0, \ldots ,n-1$ the Hecke algebra elements $t_i$ and $t'_i$ are the characteristic functions of $Is_i I$ and $I's_iI'$, respectively, where the two 
$s_i$ coincide. The statement of the lemma is that $ t_i \ast v_0= t_i'\ast v_0$ for these $i$. 
If $g\ast v$ denotes the action of $g\in G'$ on $v\in\mathrm{St}$ then  $t_i\ast v_0 = \sum_j g_j\ast v_0 $ after writing $Is_i = \cup_j g_j I$ (a disjoint sum).  Thus it suffices to that 
$I's_i I'= \cup_j g_j I$, with the same $g_j$. This is obvous for $i=0$ since $s_0$ normalizes $I$ and $I'$. For $i=1, \ldots , n-1$ the root $\alpha_i=e_i-e_{i+1}$ is a simple for 
both groups. We have 
\[ 
Is_i I = \cup_{t\in O/\varpi O} x_{\alpha_i}(t) s_i I 
\] 
and a similar decomposition for $I's_i I'$.
\end{proof}

 \begin{theorem} \label{thm global st}
  Let $\mathrm{St}$ be the Steinberg representation of  $G'$, the split orthogonal group of type $D_{n+1}$. 
  Let  $v_0$ be a  non-zero $I'$-fixed vector in $\mathrm{St}$. Then the 
 $H$-module $\mathrm{St}^{I}$ is generated by $v_0$ and isomorphic to $H\otimes_{H_0} \mathrm{\epsilon^+}$. In particular, it is projective. 
\end{theorem}
\begin{proof} By Lemma \ref{L:character} and the Frobenius reciprocity, we have a natural map  $H\otimes_{H_0} \mathrm{\epsilon^+} \rightarrow \mathrm{St}^{I}$.  
We need to show that $\mathrm{St}^{I} \cong A$, and generated by $v_0$. To that end, we use that 
\[ 
\mathrm{St}^{I} \cong (\mathrm{St})_U^{T(O)}. 
\] 
We can compute the right side using Corollary \ref{C:iwahori_vector}, which says that $\mathrm{St} \cong C_c^{\infty}(U')$. Hence $\mathrm{St}_U \cong C_c^{\infty}(U'/U)$. 
Now note that $U'/U\cong k_0^n$ as $T=(k_0^{\times})^n$-module. Under the isomorphism $\mathrm{St}_U \cong C_c^{\infty}(k_0^n)$ the vector $v_0$ is mapped to 
the characteristic function of $O^n$.  Now observe that $T/T(O)$-translates of the characteristic function of $O^n$ form a basis of $C_c^{\infty}(k_0^n)^{T(O)}$ on which 
$T/T(O)$ acts freely. In other words,  $C_c^{\infty}(k_0^n)^{T(O)}$ is isomorphic to $A$, as an $A$ module. 
\end{proof}

\section{Acknowledgment} The authors would like to thank Sol Friedberg for sharing his ideas and suggestions. In particular, the second author 
would like to thank Sol Friedberg for an invitation to give a talk at Boston College that was crucial to existence of this article.

 \end{document}